\documentclass[a4paper]{amsart}
\usepackage{amssymb}
\usepackage[mathcal]{euscript}
\usepackage[cmtip,all]{xy}
\usepackage{color}

\usepackage{tikz-cd}

\newcommand{\bydef}{:=}
\newcommand{\defby}{=:}

\newcommand{\id}{\mathrm{id}}

\newcommand{\bi}{\mathbf{i}}

\newcommand{\cA}{\mathcal{A}}
\newcommand{\cB}{\mathcal{B}}

\newcommand{\ZZ}{\mathbb{Z}}

\newcommand{\RR}{\mathbb{R}}
\newcommand{\CC}{\mathbb{C}}

\newcommand{\FF}{\mathbb{F}}
\newcommand{\KK}{\mathbb{K}}
\newcommand{\chr}[1]{\mathrm{char}\,#1}

\DeclareMathOperator{\Hom}{\mathrm{Hom}}
\DeclareMathOperator{\End}{\mathrm{End}}

\DeclareMathOperator{\im}{\mathrm{im}\,}
\DeclareMathOperator{\Aut}{\mathrm{Aut}}
\DeclareMathOperator{\AAut}{\mathbf{Aut}}

\newenvironment{romanenumerate}
 {\begin{enumerate}
 
 }{\end{enumerate}}

\DeclareMathOperator{\Ext}{\mathrm{Ext}} 

\newcommand{\Hc}{\textup{H}}  
\newcommand{\Zc}{\textup{Z}}
\newcommand{\Bc}{\textup{B}}
\newcommand{\dc}{\textup{d}}

\newtheorem{theorem}{Theorem}[section]
\newtheorem{proposition}[theorem]{Proposition}
\newtheorem{lemma}[theorem]{Lemma}
\newtheorem{corollary}[theorem]{Corollary}

\theoremstyle{definition}
\newtheorem{df}[theorem]{Definition}
\newtheorem{example}[theorem]{Example}

\theoremstyle{remark}
\newtheorem{remark}[theorem]{Remark}

\begin{document}

\title{Graded-simple algebras and cocycle twisted loop algebras}

\author{Alberto Elduque}
\address{Departamento de Matem\'{a}ticas
 e Instituto Universitario de Matem\'aticas y Aplicaciones,
 Universidad de Zaragoza, 50009 Zaragoza, Spain}
\email{elduque@unizar.es}
\thanks{Supported by grants MTM2017-83506-C2-1-P (AEI/FEDER, UE) and E22\_17R (Diputaci\'on General de Arag\'on)}




\begin{abstract}
The loop algebra construction by Allison, Berman, Faulkner, and Pianzola, describes graded-central-simple algebras with split centroid in terms of central simple algebras graded by a quotient of the original grading group. Here the restriction on the centroid is removed, at the expense of allowing some deformations (cocycle twists) of the loop algebras.
\end{abstract}

\maketitle

\section{Introduction}\label{se:intro}

The graded-central-simple algebras (not necessarily associative, nor Lie) with split centroid were shown in \cite{ABFP} to be isomorphic to loop algebras of algebras graded by a quotient group that are central simple as ungraded algebras. This is a very important reduction, as the graded-central-simple algebras may fail to be nice as ungraded algebras, for instance, they may fail to be simple or semisimple.

The purpose of this paper is to remove the restriction of the centroid being split, at the expense of allowing certain deformations of the loop algebra construction. These deformations will be based on a symmetric $2$-cocycle on the grading group with values in the multiplicative group of the ground field. 

Graded-central-simple algebras over the real field have been studied, with a different approach based on Galois descent from $\CC$ to $\RR$, in \cite{BKpr}. Their results are subsumed nicely in our more general description.

\smallskip

All the algebras considered will be defined over a ground field $\FF$, unless otherwise stated, and they are just vector spaces over $\FF$ endowed with a bilinear multiplication, usually denoted by juxtapostion. No assumption on associativity, dimension (which can be infinite), or existence of unity, is made.

For the basic facts on gradings, the reader may consult \cite{EKmon}. Here we will review some basic definition. Let $\cA$ be an algebra and $G$ a group. A \emph{$G$-grading} on $\cA$ is a vector space decomposition $\Gamma:\cA=\bigoplus_{g\in G}\cA_g$ such that $\cA_g\cA_h\subseteq \cA_{gh}$ for any $g,h\in G$. The nonzero elements in $\cA_g$ are said to be \emph{homogeneous of degree $g$}. The \emph{support} of $\Gamma$ is the set $\{g\in G\mid \cA_g\neq 0\}$.

Gradings by abelian groups often arise as eigenspace decompositions with respect to a family of commuting diagonalizable automorphisms. Over an arbitrary field, a $G$-grading $\Gamma$ on $\cA$ is equivalent to a homomorphism of affine group schemes $\eta_\Gamma:G^D\rightarrow \AAut_\FF(\cA)$, where $G^D$ is the diagonalizable group scheme represented by the group algebra $\FF G$. In this paper, only gradings by abelian groups will be considered.

Let $\Gamma: \cA=\bigoplus_{g\in G} \cA_g$ and $\Gamma':\cA'=\bigoplus_{g\in G} \cA'_g$ be two gradings by an abelian group $G$. The $G$-gradings $\Gamma$ and $\Gamma'$ are \emph{isomorphic} if $\cA$ and $\cA'$ are isomorphic as $G$-graded algebras, i.e., if  there exists an isomorphism of algebras $\varphi:\cA\rightarrow\cA'$ such that $\varphi(\cA_g)=\cA'_g$ for all $g\in G$. We will write then $\cA\simeq_G\cA'$  and say that $\cA$ and $\cA'$ are \emph{graded-isomorphic}.

\medskip

The paper is structured as follows. Section \ref{se:ext} will review known results on extensions of abelian groups and their associated $2$-symmetric cocycles needed in the sequel. These symmetric $2$-cocycles $\tau:G\times G\rightarrow \FF^\times$ are used in Section \ref{se:cocycle_twists} to define a new multiplication on any $G$-graded algebra $\cA$. The new algebra thus obtained is denoted by $\cA^\tau$ and it is associative, Lie, ..., if $\cA$ is so. The basic properties of these \emph{cocycle twists} are given. Actually, these cocycle twists are particular instances of the \emph{graded contractions} considered in \cite{MontignyPatera}, \cite{MoodyPatera}. The underlying ideas go back to \cite{Inonu-Wigner}.

Section \ref{se:loop} reviews the general loop algebra construction in \cite{ABFP} and some of its main properties. Section \ref{se:graded-central-simple} contains the main results of the paper, which show that any graded-central-simple algebra is graded-isomorphic to a cocycle twist of a loop algebra of a central simple algebra graded by a quotient of the original grading group (Theorem \ref{th:main}). Actually, the graded-isomorphism class of a $G$-graded-central-simple algebra is determined by a subgroup $H$ of $G$, an element in the symmetric second cohomology group $\Hc^2_\textup{sym}(H,\FF^\times)$, and the equivalence class of a central simple (as ungraded algebra) $G/H$-graded algebra, under a precise equivalence relation weaker than graded-isomorphism (Corollary \ref{co:less_freedom}).

The connection with the approach over the real field by Galois descent in \cite{BKpr} is explained in Section \ref{se:real}.

\bigskip

\section{$\Ext$, extensions, and $\Hc^2_{\textrm{sym}}$}\label{se:ext}

In this section, several well-known results on extensions of abelian groups will be reviewed. The reader may consult \cite[Chapters III and VI]{Hilton-Stammbach}.

Unless otherwise stated, multiplicative notation will be used for abelian groups.

\subsection{$\Ext(A,B)$ and extensions}
Given two abelian groups (i.e., $\ZZ$-modules) $A$, $B$, the abelian group $\Ext_\ZZ(A,B)$, or simply  $\Ext(A,B)$, can, and will, be identified with the set of equivalence classes of extensions  of $A$ by $B$ (in the category of abelian groups, and we will refer to them as \emph{abelian extensions}), where an extension of $A$ by $B$ is a short exact sequence $ 1\rightarrow B\rightarrow E\rightarrow A\rightarrow 1$, and two extensions
\[
1\longrightarrow B\longrightarrow E\longrightarrow A\longrightarrow 1\qquad\text{and}\qquad
1\longrightarrow B\longrightarrow E'\longrightarrow A\longrightarrow 1
\]
are equivalent if there is a homomorphism $\varphi:E\rightarrow E'$, necessarily bijective, such that the diagram
\[
\begin{tikzcd}
1\arrow[r]&B\arrow[r]\arrow[d, equal]&E\arrow[r] \arrow[d, "\varphi"]&A\arrow[r]\arrow[d, equal]&1\\
1\arrow[r]&B\arrow[r]&E'\arrow[r]&A\arrow[r]&1
\end{tikzcd}
\]
is commutative. 

For a homomorphism $f:A'\rightarrow A$, the natural homomorphism
\[
f^*=\Ext(f,B):\Ext(A,B)\longrightarrow \Ext(A',B)
\]
is obtained as follows. Let $\xi:1\rightarrow B\overset{i}\rightarrow E\xrightarrow{p} A\rightarrow 1$ be an abelian extension, and let 
$\tilde E$ be the pull-back of $p$ and $f$:
\[
\tilde E=\{(x,a')\in E\times A'\mid p(x)=f(a')\}.
\]
Then we obtain a commutative diagram with exact rows:
\[
\begin{tikzcd}
\phantom{\xi:\ }1\arrow[r]&B\arrow[r, "j"]\arrow[d, equal]&\tilde E\arrow[r, "\tilde p_2"] \arrow[d, "\tilde p_1"]&A'\arrow[r]\arrow[d, "f"]&1\\
\xi:\ 1\arrow[r]&B\arrow[r, "i"]&E\arrow[r, "p"]&A\arrow[r]&1
\end{tikzcd}
\]
where $\tilde p_1(x,a')=x$ and $\tilde p_2(x,a')=a'$ are the canonical projections, and $j(x)=\bigl(i(x),e\bigr)$ for any $x\in B$ ($e$ denotes the neutral element). The extension of $A'$ by $B$ in the first row is denoted by $\xi f$, and the map
\begin{equation}\label{eq:f*}
\begin{split}
f^*:\Ext(A,B)&\longrightarrow \Ext(A',B)\\
 [\xi]\ &\mapsto\ [\xi f]
\end{split}
\end{equation}
is well defined and equals $\Ext(f,B)$.

\bigskip

\subsection{$\Ext(A,B)\simeq {\normalfont \Hc^2_\textrm{sym}}(A,B)$}
\null\quad On the other hand, the set of equivalence classes of central extensions, in the category of groups, of the  group $A$ by the abelian group $B$, that is, equivalence classes as above, but of short exact sequences of  groups $1\rightarrow B\overset{i}\rightarrow E\xrightarrow{p} A\rightarrow 1$ with $i(B)$ central in $E$, can be identified with the second cohomology group $\Hc^2(A,B)=\Zc^2(A,B)/\Bc^2(A,B)$, where
\begin{multline*}
\Zc^2(A,B)=\{\sigma:A\times A\rightarrow B\mid\\
 \sigma(a_1,a_2)\sigma(a_1a_2,a_3)=\sigma(a_1,a_2a_3)\sigma(a_2,a_3)\ \forall a_1,a_2,a_3\in A\}
\end{multline*}
is the set of \emph{$2$-cocycles}, and
\[
\Bc^2(A,B)=\{ \dc \gamma\mid \gamma:A\rightarrow B\ \text{a map}\}
\]
is the set of $2$-coboundaries, where $\dc \gamma(a_1,a_2)=\gamma(a_1)\gamma(a_2)\gamma(a_1a_2)^{-1}$ for any map $\gamma:A\rightarrow B$ and $a_1,a_2\in A$.

The element in $\Hc^2(A,B)$ that corresponds to the central extension $\xi:1\rightarrow B\overset{i}\rightarrow E\xrightarrow{p} A\rightarrow 1$ is obtained by fixing a section $s:A\rightarrow E$ of $p$ ($s$ is not a homomorphism in general). Then for any $a_1,a_2\in A$, the element $s(a_1)s(a_2)s(a_1a_2)^{-1}$ is in $\ker p=\im i$, so there is a unique element $\sigma(a_1,a_2)\in B$ such that
\begin{equation}\label{eq:isigma}
i\bigl(\sigma(a_1,a_2)\bigr)=s(a_1)s(a_2)s(a_1a_2)^{-1}.
\end{equation}
Then $\sigma:A\times A\rightarrow B$ is a $2$-cocycle whose equivalence class $[\sigma]$ in $\Hc^2(A,B)$ does not depend on the chosen section $s$. This equivalence class $[\sigma]$ is the element in $\Hc^2(A,B)$ that corresponds to the equivalence class of $\xi$.

Moreover, for $A$ and $B$ abelian, $\xi$ is an abelian extension if and only if $\sigma$ is symmetric. 

Denote by $\Zc^2_{\textup{sym}}(A,B)$ the subgroup of $\Zc^2(A,B)$ of the symmetric $2$-cocycles, that is, $2$-cocycles $\sigma$ such that $\sigma(a_1,a_2)=\sigma(a_2,a_1)$ for any $a_1,a_2\in A$, and note that $\Bc^2(A,B)$ is contained in $\Zc^2_\textup{sym}(A,B)$. Then, for $A$ and $B$ abelian, $\Ext(A,B)$ can be identified too with the quotient
\[
\Hc^2_\textup{sym}(A,B)=\Zc^2_\textup{sym}(A,B)/\Bc^2(A,B).
\]

Given an abelian extension $\xi:1\rightarrow B\overset{i}\rightarrow E\xrightarrow{p} A\rightarrow 1$, the $2$-cocycle $\sigma\in\Zc^2_\textup{sym}(A,B)$ defined in \eqref{eq:isigma}, and a homomorphism of abelian group $f:A'\rightarrow A$, the element in $\Hc^2_\textup{sym}(A',B)$ attached to $\xi f$ is just $[\sigma\circ(f\times f)]$. In other words, $f^*$ in \eqref{eq:f*} corresponds to the natural map (also denoted by $f^*$):
\begin{equation}\label{eq:f*H2}
\begin{split}
f^*: \Hc^2_\textup{sym}(A,B)&\longrightarrow\Hc^2_\textup{sym}(A',B)\\
[\sigma]\ &\mapsto\ [\sigma\circ(f\times f)]
\end{split}
\end{equation}

\bigskip

\subsection{The long exact sequence}
Given an abelian group $G$, a subgroup $H$, and the associated quotient group $G/H$, consider the corresponding short exact sequence
\begin{equation}\label{eq:HiotaGpi}
\zeta: 
\begin{tikzcd}
1\arrow[r]&H\arrow[r, hookrightarrow, "\iota"]&G\arrow[r, "\pi"]&G/H\arrow[r]&1
\end{tikzcd}
\end{equation}
where $\iota$ is the inclusion map and $\pi$ the canonical projection. Given any abelian group $F$, there is the associated long exact sequence
\begin{multline}\label{eq:long}
1\rightarrow \Hom(G/H,F)\xrightarrow{\pi^*}\Hom(G,F)\xrightarrow{\iota^*}\Hom(H,F)\\
\xrightarrow{\ \delta\ }\Ext(G/H,F)\xrightarrow{\pi^*}\Ext(G,F)\xrightarrow{\iota^*}\Ext(H,F)\rightarrow 1
\end{multline}
because $\Ext^n(.,.)$ is trivial for $n\geq 2$ (see, e.g., \cite[Lemma 3.3.1]{Weibel}).
The maps $\pi^*$ and $\iota^*$ in the first row are just the precompositions by $\pi$ and $\iota$, respectively. The maps $\pi^*$ and $\iota^*$ in the second row are given by \eqref{eq:f*}. As for the connecting homomorphism $\delta$, it is obtained as follows. Given a homomorphism $f:H\rightarrow F$, let $E$ be the push-out of $f$ and $\iota$:
\[
E=F\times G / \left\langle \bigl(f(x),\iota(x)^{-1}\bigr)\mid x\in H\right\rangle.
\]
There appears a commutative diagram with exact rows:
\[
\begin{tikzcd}
\zeta:\ 1\ar[r] &H\ar[r, hookrightarrow, "\iota"]\ar[d, "f"]&G\ar[r, "\pi"]\ar[d, "j_2"]&G/H\ar[r]\ar[d, equal] &1\\
\phantom{\zeta:\ }1\ar[r]&F\ar[r, "j_1"]&E\ar[r,"p"]&G/H\ar[r]&1
\end{tikzcd}
\]
where $j_1$ and $j_2$ are the canonical homomorphisms and $p([x,g])=\pi(g)$ for any $[x,g]\in E$ ($[x,g]$ denotes the class of the element $(x,g)$ modulo the subgroup $\left\langle \bigl(f(x),\iota(x)^{-1}\bigr)\mid x\in H\right\rangle$). The extension in the second row is denoted by $f\zeta$, and the map
\[
\begin{split}
\delta:\Hom(H,F)&\longrightarrow \Ext(G/H,F)\\
f\quad &\mapsto\qquad [f\zeta]
\end{split}
\]
is precisely the connecting homomorphism.

Given a section $s:G/H\rightarrow G$, the element in $\Hc^2_\textup{sym}(G/H,H)$ corresponding to $\zeta$ is the class $[\sigma]$ of the $2$-cocycle
\begin{equation}\label{eq:sigma}
\begin{split}
\sigma:G/H\times G/H&\longrightarrow H\\
(\bar g_1,\bar g_2)\ &\mapsto\ s(\bar g_1)s(\bar g_2)s(\bar g_1\bar g_2)^{-1}.
\end{split}\end{equation}
Then $\tilde s:G/H\rightarrow E$, $\bar g\mapsto [e,s(\bar g)]$ is a section of $p$ in $f\zeta$, and for $\bar g_1,\bar g_2\in G/H$
\[
\tilde s(\bar g_1)\tilde s(\bar g_2)\tilde s(\bar g_1\bar g_2)^{-1}=[(e,\iota\bigl(\sigma(\bar g_1,\bar g_2)\bigr)]
=[f\bigl(\sigma(\bar g_1,\bar g_2)\bigr),e].
\]
Hence the connecting homomorphim $\delta$ corresponds to the map (denoted too by $\delta$):
\begin{equation}\label{eq:deltaH2}
\begin{split}
\delta:\Hom(H,F)&\longrightarrow \Hc^2_\textup{sym}(G/H,F)\\
 f\ &\mapsto\quad [f\circ\sigma]\,.
\end{split}
\end{equation}

The long exact sequence \eqref{eq:long} can then be substituted by
\begin{multline}\label{eq:longH}
1\rightarrow \Hom(G/H,F)\xrightarrow{\pi^*}\Hom(G,F)\xrightarrow{\iota^*}\Hom(H,F)\\
\xrightarrow{\,\delta\,}\Hc^2_\textup{sym}(G/H,F)\xrightarrow{\pi^*}\Hc^2_\textup{sym}(G,F)\xrightarrow{\iota^*}\Hc^2_\textup{sym}(H,F)\rightarrow 1
\end{multline}
with $\delta$ in \eqref{eq:deltaH2} and $\pi^*$ and $\iota^*$ in the second row as in \eqref{eq:f*H2}.

The exactness of \eqref{eq:long} has the following consequence that will be critical later on:

\begin{proposition}\label{pr:tau_extension}
Let $G$ and $F$ be abelian groups, $H$ a subgroup of $G$, and $\tau':H\times H\rightarrow F$ a symmetric $2$-cocycle. Then there is a symmetric $2$-cocycle $\tau\in \Zc^2_\textup{sym}(G,F)$ that extends $\tau'$ (i.e., $\tau'=\tau\vert_{H\times H}$).
\end{proposition}
\begin{proof}
The homomorphism $\iota^*:\Hc^2_\textup{sym}(G,F)\rightarrow\Hc^2_\textup{sym}(H,F)$ being surjective, there is a $2$-cocycle $\tilde\tau\in \Zc^2_\textup{sym}(G,F)$ such that $[\tilde\tau\vert_{H\times H}]=\iota^*([\tilde\tau])=[\tau']$, so there is a map $\gamma:H\rightarrow F$ such that $\tau'=\bigl(\tilde\tau\vert_{H\times H})(\dc \gamma)$. That is, 
\[
\tau'(h_1,h_2)=\tilde\tau(h_1,h_2)\gamma(h_1)\gamma(h_2)\gamma(h_1h_2)^{-1}
\] 
for any $h_1,h_2\in H$. Extend $\gamma$ to a map $\tilde \gamma:G\rightarrow F$ (for instance, with $\tilde\gamma(g)=e$ for any $g\in G\setminus H$). Then $\tau=\tilde\tau(\dc\tilde\gamma)$ satisfies $\tau\vert_{H\times H}=\tau'$.
\end{proof}

\bigskip

\section{Cocycle twists}\label{se:cocycle_twists}

\begin{df}
Let $G$ be an abelian group, let $\cA$ be an algebra over $\FF$ endowed with a $G$-grading: $\cA=\bigoplus_{g\in G}\cA_g$, and let $\tau:G\times G\rightarrow \FF^\times$ be a symmetric $2$-cocycle.
Define a new multiplication on $\cA$ by the formula
\begin{equation}\label{eq:tau_twist}
x*y\bydef \tau(g_1,g_2)xy
\end{equation}
for $g_1,g_2\in G$, $x\in \cA_{g_1}$, $y\in\cA_{g_2}$.

The new algebra thus defined will be called the \emph{$\tau$-twist} of $\cA$, and will be denoted by $\cA^\tau$.
\end{df}

\begin{remark}
In \cite{MoodyPatera} (see also \cite{MontignyPatera}) a more general situation is considered, where $\tau:G\times G\rightarrow \FF$ is allowed to take the value $0$, but still being symmetric and satisfying the cocycle condition: $\tau(g_1,g_2)\tau(g_1g_2,g_3)=\tau(g_1,g_2g_3)\tau(g_2,g_3)$ for $g_1,g_2,g_3\in G$. The resulting \emph{twisted} algebra is said to be a \emph{graded contraction} of $\cA$ and the interest in the above mentioned references lies in those $\tau$'s that indeed take the value $0$. In this way one can obtain, for example, solvable Lie algebras as graded contractions of semisimple Lie algebras, as some nonzero structure constants in the original algebra may become $0$ in the twist.
\end{remark}

\begin{example}\label{ex:C} 
Consider the real algebra of complex numbers $\CC$, graded by the cyclic group of order $2$: $C_2=\{e,h\}$, so $\CC_e=\RR 1$, $\CC_h=\CC\bi$. Let $\tau:C_2\times C_2\rightarrow \RR^\times$ be the symmetric $2$-cocycle such that $\tau(e,e)=\tau(e,h)=1$ and $\tau(h,h)=-1$. Then in the $\tau$-twist $\CC^\tau$ we have
\[
1*1=1, \quad 1*\bi=\bi*1=\bi,\quad\text{but}\quad \bi*\bi= 1,
\]
so that $\CC^\tau$ is isomorphic to the group algebra $\RR C_2$ (isomorphic to $\RR\times\RR$).
\end{example}

\begin{example}\label{ex:twisted_group_algebra}
For any abelian group $G$ and symmetric $2$-cocycle $\tau\in\Zc^2_\textup{sym}(G,\FF^\times)$, the $\tau$-twist $\bigl(\FF G\bigr)^\tau$ of the group algebra $\FF G$ is denoted traditionally by $\FF^\tau G$. 

This makes sense for not necessarily symmetric $2$-cocycles and nonabelian groups. The algebras $\FF^\tau G$ are called \emph{twisted group algebras}, and they play a key role in Schur's theory of projective representations of finite groups. Twisted group algebras are, up to isomorphism, the graded-division algebras with one-dimensional homogeneous components. Moreover, for $\tau_1,\tau_2\in \Zc^2(G,\FF^\times)$, $\FF^{\tau_1}G\simeq_G \FF^{\tau_2}G$ if and only if $[\tau_1]=[\tau_2]$ in $\Hc^2(G,\FF^\times)$.

Natural examples of twisted group algebras are quaternion algebras over a field $\FF$, $\chr\FF\neq 2$, which can be described, up to isomorphism, as the twisted group algebras $\FF^\tau G$, with $G=C_2\times C_2$ and $\tau$ a \emph{not symmetric} $2$-cocycle.
\end{example}

\medskip

Some basic properties of cocycle twists are given in the next result.

\begin{proposition}\label{pr:cocycle_twists}
Let $G$ be an abelian group and let $\cA=\bigoplus_{g\in G}\cA_g$ be a $G$-graded algebra over $\FF$.
\begin{romanenumerate}
\item For $\sigma,\tau\in\Zc^2_\textup{sym}(G,\FF^\times)$, $\bigl(\cA^\sigma\bigr)^\tau=\cA^{\sigma\tau}$.

\item If $\tau\in\Bc^2(G,\FF^\times)$, then $\cA^\tau\simeq_G\cA$. More generally, if $\tau_1,\tau_2\in \Zc^2_\textup{sym}(G,\FF^\times)$ and $[\tau_1]=[\tau_2]$ in $\Hc^2_\textup{sym}(G,\FF^\times)$, then $\cA^{\tau_1}\simeq_G\cA^{\tau_2}$. In other words, for $\tau\in\Zc^2_\textup{sym}(G,\FF^\times)$, the $G$-graded isomorphism class of $\cA^\tau$ depends only on $[\tau]\in\Hc^2_\textup{sym}(G,\FF^\times)$.

\item If $\overline{\FF}$ is an algebraic closure of $\FF$ and $\tau\in\Zc^2_\textup{sym}(G,\FF^\times)$, then $\cA^\tau\otimes_\FF\overline{\FF}\simeq_G\cA\otimes_\FF\overline{\FF}$. In particular, if $\cA$ is an asociative, alternative, Lie, linear Jordan, ...., algebra, so is $\cA^\tau$.

\item If $\cA$ is graded-simple, so is $\cA^\tau$.
\end{romanenumerate}
\end{proposition}
\begin{proof}
Part (i) is clear. For (ii), let $\gamma:G\rightarrow\FF^\times$ be a map such that $\tau_1=\tau_2(\dc\gamma)$. The multiplication in $\cA^{\tau_1}$ and $\cA^{\tau_2}$ are  $x*y=\tau_1(g_1,g_2)xy$ and $x\star y=\tau_2(x,y)xy$, respectively, for $g_1,g_2\in G$, $x\in \cA_{g_1}$, $y\in\cA_{g_2}$. Then the linear automorphism $\varphi\in\End_\FF(\cA)$ such that $\varphi(x)=\gamma(g)x$ for $g\in G$ and $x\in\cA_g$ preserves the grading and satisfies:
\[
\begin{split}
\varphi(x)\star\varphi(y)&=\gamma(g_1)\gamma(g_2)\tau_2(g_1,g_2)xy\\
&=\gamma(g_1g_2)(\dc\gamma)(g_1,g_2)\tau_2(g_1,g_2)xy\\
&=\gamma(g_1g_2)\tau_1(g_1,g_2)xy\\
&=\gamma(g_1g_2)x*y=\varphi(x*y).
\end{split}
\]
For (iii) note that for $\tau\in\Zc^2_\textup{sym}(G,\FF^\times)\subseteq \Zc^2_\textup{sym}(G,\overline{\FF}^\times)$, $\cA^\tau\otimes_\FF\overline{\FF}=\bigl(\cA\otimes_\FF\overline{\FF}\bigr)^\tau$. But $\overline{\FF}^\times$ is divisible (hence injective as a $\ZZ$-module), so $\Ext(G,\overline{\FF}^\times)\simeq \Hc^2_\textup{sym}(G,\overline{\FF}^\times)=1$, and hence, because of (ii), 
$\bigl(\cA\otimes_\FF\overline{\FF}\bigr)^\tau\simeq_G\cA\otimes_\FF\overline{\FF}$.

As for (iv), it is enough to note that the ideal generated by any homogeneous element $x\in\cA_g$ is the same in $\cA$ or in $\cA^\tau$.
\end{proof}

\begin{example}\label{ex:chi_sigma} 
Let $G$ be an abelian group, $H$ a subgroup, and $\overline{G}=G/H$ the corresponding quotient. Let $\cA$ be a $\overline{G}$-graded algebra and let $\chi\in\Hom(H,\FF^\times)$ (a character on $H$). Finally, let $s:\overline{G}\rightarrow G$ be a section of the canonical projection $\pi:G\rightarrow \overline{G}$. Then in \cite[Definition 6.3.1]{ABFP}, the algebra $\cA_\chi$ is defined on the same vector space as $\cA$, but with new multiplication
\[
x\cdot_\chi y=\chi\Bigl(s(\overline{g_1})s(\overline{g_2})s(\overline{g_1}\overline{g_2})^{-1}\Bigr)xy
\]
for $\overline{g}_1,\overline{g}_2\in \overline{G}$, $x\in \cA_{\overline{g_1}}$, $y\in\cA_{\overline{g_2}}$. We see that $\cA_\chi$ is the $(\chi\circ\sigma)$-twist $\cA^{\chi\circ\sigma}$ of $\cA$, where $\sigma\in\Zc^2_\textup{sym}(\overline{G},H)$ is the symmetric $2$-cocycle considered in \eqref{eq:sigma}. Moreover, note that $[\chi\circ\sigma]=\delta(\chi)$, where $\delta$ is the connecting homomorphism in \eqref{eq:deltaH2} for $F=\FF^\times$.
\end{example}

\smallskip

Recall that the \emph{centroid} $C(\cA)$ of an algebra $\cA$ is the unital associative subalgebra of $\End_\FF(\cA)$ given by
\[
C(\cA)=\{c\in\End_\FF(\cA)\mid c(xy)=c(x)y=xc(y)\ \forall x,y\in\cA\}\,.
\]
If the algebra $\cA$ is simple, then $C(\cA)$ is a field extension of $\FF$. The algebra $\cA$ is said to be \emph{central simple} if it is simple and this field extension is trivial: $C(\cA)=\FF 1$. (Here $1$ denotes the identity map on $\cA$.)  Any simple algebra is central simple when considered as an algebra over its centroid.

Given an abelian group $G$ and a $G$-graded algebra $\cA=\bigoplus_{g\in G}\cA_g$, consider the subspaces
\[
C(\cA)_g=\{c\in C(\cA)\mid c\cA_{g'}\subseteq \cA_{gg'}\ \forall g'\in G\}
\]
for $g\in G$. Their sum $\bigoplus_{g\in G}C(\cA)_g$ is direct, but it may fail to be the whole $C(\cA)$.
The $G$-graded algebra $\cA$ is called \emph{graded-central} if $C(\cA)_e=\FF 1$.

Also, the $G$-graded algebra $\cA$ is said to be \emph{graded-simple} if $\cA^2=\cA$ and $\cA$ does not contain any proper graded ideals, and \emph{graded-central-simple} if it is graded-simple and graded-central.

If $\cA$ is graded-simple, then $C(\cA)=\bigoplus_{g\in G}C(\cA)_g$ (\cite[Proposition 2.16]{BenkartNeher}) and it is a \emph{graded-field}, i.e., it is commutative and any nonzero homogeneous element is invertible. If $H$ is the support of the $G$-grading on $C(\cA)$: $H=\{g\in G\mid C(\cA)_g\neq 0\}$, $H$ is a subgroup of $G$. Then if $\KK=C(\cA)_e$ (a field extension of $\FF$), $C(\cA)$ is isomorphic, as a graded algebra, to a twisted group algebra $\KK^\tau H$ (see Example \ref{ex:twisted_group_algebra}), for a symmetric $2$-cocycle $\tau\in\Zc^2_\textup{sym}(H,\KK^\times)$, the $G$-grading (with support $H$) in $\KK^\tau H$ being the natural one. Moreover, $\cA$ is graded-central-simple when considered as an algebra over $\KK$.

The next results shows the behavior of the centroid of a $G$-graded-central simple algebra under cocycle twists.

\begin{proposition}\label{pr:centroid_twists}
Let $G$ be an abelian group and $\cB$ a $G$-graded-central-simple algebra. For any $\tau\in\Zc^2_\textup{sym}(G,\FF^\times)$, $C(\cB^\tau)\simeq_G C(\cB)^\tau$.
\end{proposition}
\begin{proof}
For any homogeneous element $c\in C(\cB)_h$, define $c^\tau\in\End_\FF(\cB)$ by
\[
c^\tau(x)=\tau(h,g)c(x)
\]
for $g\in G$ and $x\in\cB_g$. Denote by $*$ the multiplication in $\cB^\tau$: $x*y=\tau(g_1,g_2)xy$, for $g_1,g_2\in G$, $x\in\cB_{g_1}$, $y\in\cB_{g_2}$. Then, under these assumptions,
\[
\begin{split}
c^\tau(x*y)&=\tau(g_1,g_2)c^\tau(xy)\\
  &=\tau(g_1,g_2)\tau(h,g_1g_2)c(xy)\\
  &=\tau(h,g_1)\tau(hg_1,g_2)c(x)y\\
  &=\tau(hg_1,g_2)c^\tau(x)y=c^\tau(x)*y
\end{split}
\]
and also, because of the symmetry of $\tau$:
\[
\begin{split}
c^\tau(x*y)&=\tau(g_1,g_2)\tau(h,g_1g_2)c(xy)\\
  &=\tau(h,g_2)\tau(g_1,hg_2)xc(y)=x*c^\tau(y).
\end{split}
\]
Thus $c^\tau\in C(\cB^\tau)_h$. Besides, if $c_1\in C(\cB)_{h_1}$ and $c_2\in C(\cB)_{h_2}$, then for $x\in \cB_g$ we have:
\[
\begin{split}
c_1^\tau c_2^\tau(x)&=\tau(h_1,h_2g)\tau(h_2,g)c_1c_2(x)\\
  &=\tau(h_1,h_2)\tau(h_1h_2,g)c_1c_2(x)\\
   &=\tau(h_1,h_2)(c_1c_2)^\tau (x),
\end{split}
\]
so $c_1^\tau c_2^\tau=\tau(h_1,h_2)(c_1c_2)^\tau$, and the map $C(\cB)^\tau\rightarrow C(\cB^\tau)$, $c\mapsto c^\tau$, for homogeneous $c$, is an isomorphism.
\end{proof}

\bigskip

\section{Loop algebras}\label{se:loop}

Given an abelian group $G$, a subgroup $H$, the canonical projection $\pi:G\rightarrow \overline{G}=G/H$, $g\mapsto \pi(g)=\overline{g}$, and an algebra $\cA$ graded by $\overline{G}$: $\cA=\bigoplus_{\overline{g}\in\overline{G}}\cA_{\overline{g}}$, the \emph{loop algebra} $L_\pi(\cA)$ is the $G$-graded algebra
\[
L_\pi(\cA)=\bigoplus_{g\in G}\cA_{\overline{g}}\otimes g
\]
which is a subalgebra of the tensor product $\cA\otimes_\FF\FF G$, where $\FF G$ is the group algebra of $G$ (see \cite[Definition 3.1.1]{ABFP}).

\smallskip

If $\cA$ is $G$-graded-central-simple, then its centroid $C(\cA)$ is said to be \emph{split} (see \cite[Definition 4.3.6]{ABFP}) if $C(\cA)$ is isomorphic, as a graded algebra, to the (untwisted) group algebra: $C(\cA)\simeq_G\FF H$. This is always the case if $\FF$ is algebraically closed by Proposition \ref{pr:cocycle_twists}.

The main results in \cite{ABFP} reduce the study of the graded-central-simple algebras with split centroid to the study of the central simple graded algebras (i.e., graded algebras which are central simple as (ungraded) algebras).

\begin{theorem}\label{th:ABFP} \textup{(see \cite[Theorem 7.1.1]{ABFP})}\quad
Let $G$ be an abelian group, $H$ a subgroup of $G$, and $\pi:G\rightarrow \overline{G}=G/H$ the canonical projection.
\begin{enumerate}
\item If $\cA$ is a central simple algebra graded by $\overline{G}$, then the loop algebra $L_\pi(\cA)$ is a $G$-graded-central-simple algebra, and the map
\[
\begin{split}
\FF H&\longrightarrow C\bigl(L_\pi(\cA)\bigr)\\
 h\, &\mapsto \bigl(x\otimes g\mapsto x\otimes hg)
\end{split}
\]
for $g\in G$, $x\in \cA_{\pi(g)}$, is an isomorphism of $G$-graded algebras. (Hence $C\bigl(L_\pi(\cA)\bigr)\simeq_G\FF H$.)

\item if $\cB$ is a $G$-graded-central-simple algebra with split centroid $C(\cB)\simeq_G\FF H$, then there exists a central simple and $\overline{G}$-graded algebra $\cA$ such that $\cB\simeq_G L_\pi(\cA)$.

\item If $\cA$ and $\cA'$ are central simple and $\overline{G}$-graded algebras, then $L_\pi(\cA)\simeq_G L_\pi(\cA')$ if and only if there is a character $\chi\in\Hom(H,\FF^\times)$ such that $\cA'\simeq \cA_\chi=\cA^{\chi\circ\sigma}$ (as in Example \ref{ex:chi_sigma}).
\end{enumerate}
\end{theorem}

\bigskip

\section{Graded-central-simple algebras}\label{se:graded-central-simple}

The goal in this section is to remove the restriction on the centroid of a graded-central-simple algebra to be split, at the expense of allowing cocycle twists of loop algebras.

For simplicity, given an abelian group $G$, a subgroup $H$, a $\overline{G}=G/H$-graded algebra algebra $\cA$, and a symmetric $2$-cocycle $\tau\in\Zc^2_\textup{sym}(G,\FF^\times)$, the $\tau$-twist $\bigl(L_\pi(\cA)\bigr)^\tau$ will be denoted by $L_\pi^\tau(\cA)$ and called a \emph{cocycle twisted loop algebra}.

\begin{remark}\label{re:Lpitau}
$L_\pi^\tau(\cA)$ is the subalgebra $\bigoplus_{g\in G}\cA_{\overline{g}}\otimes g$ of the tensor product $\cA\otimes_\FF\FF^\tau G$, where $\FF^\tau G$ is the twisted loop algebra, as in Example \ref{ex:twisted_group_algebra}.
\end{remark}

The next Theorem is the main result of the paper. It reduces the classification of graded-central-simple algebras, up to isomorphism, to the classification of central simple and graded algebras, up to isomorphism.

\pagebreak[2]

\begin{theorem}\label{th:main}
Let $G$ be an abelian group. 
\begin{enumerate}
\item If $\cB$ is a $G$-graded-central-simple algebra, then there is a subgroup $H$ of $G$, a central simple and $\overline{G}=G/H$-graded algebra $\cA$, and a symmetric $2$-cocycle $\tau\in\Zc^2_\textup{sym}(G,\FF^\times)$ such that $\cB\simeq_GL_\pi^\tau(\cA)$. (As usual, $\pi$ denotes the canonical projection $G\rightarrow \overline{G}$.)

\item Conversely, let $H$ be a subgroup of $G$, $\cA$ a central simple and $\overline{G}$-graded algebra, and let $\tau\in\Zc^2_\textup{sym}(G,\FF^\times)$. Then $L_\pi^\tau(\cA)$ is $G$-graded-central-simple and $C\bigl(L_\pi^\tau(\cA)\bigr)\simeq_G\FF^{\tau'}H$, where $\tau'=\tau\vert_{H\times H}$.

\item For $i=1,2$, let $H_i$ be a subgroup of $G$, $\cA_i$ a central simple and $G/H_i$-graded algebra, $\tau_i\in\Zc^2_\textup{sym}(G,\FF^\times)$. Denote by $\pi_i:G\rightarrow\overline{G_i}=G/H_i$ the canonical projection, $i=1,2$. Then $L_{\pi_1}^{\tau_1}(\cA_1)\simeq_GL_{\pi_2}^{\tau_2}(\cA_2)$ if and only if the following conditions are satisfied:
\begin{itemize}
\item $H_1=H_2\defby H$, so $\pi_1=\pi_2\defby \pi:G\rightarrow\overline{G}=G/H$, and 
\item there is a $2$-cocycle $\mu\in\Zc^2_\textup{sym}(\overline{G},\FF^\times)$ such that $[\tau_1]=\pi^*([\mu])[\tau_2]$ in $\Hc^2_\textup{sym}(G,\FF^\times)$ and $\cA_1^\mu\simeq_{\overline{G}}\cA_2$.
\end{itemize}
\end{enumerate}
\end{theorem}

\begin{proof}
For (1), if $\cB$ is a $G$-graded-central-simple algebra, and $H$ is the support of its centroid, 
then $C(\cB)$ is a twisted group algebra: $C(\cB)\simeq_G \FF^{\tau'}H$, for a $2$-cocycle 
$\tau'\in\Zc^2_\textup{sym}(H,\FF^\times)$. Proposition \ref{pr:tau_extension} shows that there is 
a $2$-cocycle $\tau\in\Zc^2_\textup{sym}(G,\FF^\times)$ such that $\tau\vert_{H\times H}=\tau'$. 
Then by Proposition \ref{pr:centroid_twists}, $C(\cB^{\tau^{-1}})\simeq_G C(\cB)^{\tau^{-1}}\simeq_G\bigl(\FF^{\tau'}H\bigr)^{\tau^{-1}}=\FF H$, 
so $C(\cB^{\tau^{-1}})$ is split. \cite[Theorem 7.1.1.(iii)]{ABFP} shows that there is a central simple 
$\overline{G}$-graded algebra $\cA$ such that $\cB^{\tau^{-1}}\simeq_G L_\pi(\cA)$, and hence 
$\cB=\bigl(\cB^{\tau^{-1}})^\tau\simeq_G L_\pi^\tau(\cA)$ by Proposition \ref{pr:cocycle_twists}.

Now (2) follows since $\bigl(L_\pi^\tau(\cA)\bigr)^{\tau^{-1}}=L_\pi(\cA)$ is $G$-graded-central-simple, and hence so is $L_\pi^\tau(\cA)$ by Propositions \ref{pr:cocycle_twists} and \ref{pr:centroid_twists}, which also imply the last part of (2).

Finally, if $L_{\pi_1}^{\tau_1}(\cA_1)\simeq_G L_{\pi_2}^{\tau_2}(\cA_2)$, then their centroids are
graded-isomorphic too, and hence with equal supports. Thus $H_1=H_2$ (see Theorem \ref{th:ABFP}.(1)). 
By part (2), $\FF^{\tau_1'}H\simeq_G\FF^{\tau_2'}H$, where $\tau_i'=\tau_i\vert_{H\times H}$, $i=1,2$, 
so that $[\tau_1']=[\tau_2']$ in $\Hc^2_\textup{sym}(H,\FF^\times)$, that is, 
$\iota^*([\tau_1])=\iota^*([\tau_2])$, where $\iota:H\hookrightarrow G$ is the inclusion. By the exactness of \eqref{eq:longH}, there is a $2$-cocycle 
$\nu\in \Zc^2_\textup{sym}(\overline{G},\FF^\times)$ such that $[\tau_1]=\pi^*([\nu])[\tau_2]$. 

 From $L_{\pi}^{\tau_1}(\cA_1)\simeq_G L_{\pi}^{\tau_2}(\cA_2)$ we get $L_{\pi}^{\tau_1\tau_2^{-1}}(\cA_1)\simeq_G L_{\pi}(\cA_2)$. But $[\tau_1\tau_2^{-1}]=\pi^*([\nu])=[\nu\circ(\pi\times\pi)]$, so with $\widehat{\nu}=\nu\circ(\pi\times\pi)$ we get 
 $L_\pi^{\tau_1\tau_2^{-1}}(\cA_1)\simeq_G L_\pi^{\widehat{\nu}}(\cA_1)=L_\pi(\cA_1^\nu)$. 
Hence $L_\pi(\cA_1^\nu)\simeq_G L_\pi(\cA_2)$, and we conclude from \cite[Theorem 7.1.1.(iii)]{ABFP} that there exists a character $\chi\in\Hom(H,\FF^\times)$ such that 
$\bigl(\cA_1^\nu\bigr)_\chi\simeq_{\overline{G}} \cA_2$. Example \ref{ex:chi_sigma} shows that 
$\bigl(\cA_1^\nu\bigr)_\chi=\cA_1^{\nu(\chi\circ\sigma)}$ and 
$[\nu(\chi\circ\sigma)]=[\nu]\delta(\chi)$. With $\mu=\nu(\chi\circ\sigma)$, which lies 
in $\Zc^2_\textup{sym}(\overline{G},\FF^\times)$, we have $\cA_1^\mu\simeq_{\overline{G}}\cA_2$, 
and since $\pi^*([\mu])=\pi^*([\nu])\pi^*\delta(\chi)=\pi^*([\nu])$ by the exactness of \eqref{eq:longH},
we obtain $[\tau_1]=\pi^*([\nu])[\tau_2]=\pi^*([\mu])[\tau_2]$, as required. 
 
The converse is clear, because from $\cA_1^\mu\simeq_{\overline{G}}\cA_2$ we get $L_\pi(\cA_1^\mu)\simeq_G L_\pi(\cA_2)$, so $L_\pi^{\widehat{\mu}}(\cA_1)\simeq_G L_\pi(\cA_2)$, where $\widehat{\mu}=\mu\circ(\pi\times\pi)$. But $[\widehat{\mu}]=\pi^*([\mu])=[\tau_1\tau_2^{-1}]$, so Proposition \ref{pr:cocycle_twists} gives $L_\pi^{\tau_1\tau_2^{-1}}(\cA_1)\simeq_G L_\pi^{\widehat{\mu}}(\cA_1)\simeq_G L_\pi(\cA_2)$, and hence $L_\pi^{\tau_1}(\cA_1)=\bigl(L_\pi^{\tau_1\tau_2^{-1}}(\cA_1)\bigr)^{\tau_2}\simeq_G L_\pi^{\tau_2}(\cA_2)$.
\end{proof}

\begin{remark}\label{re:freedom} 
There is a great freedom in choosing $\tau\in\Zc^2_\textup{sym}(G,\FF^\times)$ in the proof above, the only required condition being that $\tau\vert_{H\times H}$ should be cohomologous to $\tau'$: $\iota^*([\tau])=[\tau']$ in $\Hc^2_\textup{sym}(H,\FF^\times)$.
\end{remark}

We can express the above result in a concise way as follows. Given the abelian group $G$, consider the 
set $\overline{\mathfrak{B}}(G,\FF)$ of the isomorphism classes (as $G$-graded algebras) of 
$G$-graded-central-simple algebras, denoting by $[\cB]$ the class of an algebra $\cB$. Consider too 
the set $\overline{\mathfrak{A}}(G,\FF)$ consisting of the equivalence classes of triples 
$(H,[\tau],\cA)$, where $H$ is a subgroup of $G$, $[\tau]\in\Hc^2_\textup{sym}(G,\FF^\times)$, and 
$\cA$ is a central simple and $G/H$-graded algebra, the equivalence relation being given by
\[
(H_1,[\tau_1],\cA_1)\sim (H_2,[\tau_2],\cA_2)
\]
if $H_1=H_2(\defby H)$, $\iota^*([\tau_1])=\iota^*([\tau_2])$, ($\iota:H\hookrightarrow G$ being the
inclusion map), and if there is a $\mu\in\Zc^2_\textup{sym}(G/H,\FF^\times)$ such that 
$[\tau_1]=\pi^*([\mu])[\tau_2]$ and $\cA_1^\mu\simeq_{G/H}\cA_2$.

\begin{corollary}\label{co:main}
The map
\[
\begin{split}
\overline{\mathfrak{A}}(G,\FF)&\longrightarrow \overline{\mathfrak{B}}(G,\FF)\\
[(H,[\tau],\cA)]\ &\mapsto\ [L_\pi^\tau(\cA)]
\end{split}
\]
is a bijection.
\end{corollary}

The inverse is given by $[\cB]\mapsto [(H,[\tau],\cA)]$, where $H$ is the support of $C(\cB)$, $[\tau]\in\Hc^2_\textup{sym}(G,\FF^\times)$ is such that $\iota^*([\tau])\in\Hc^2_\textup{sym}(H,\FF^\times)$ determines the isomorphism class of $C(\cB)$ as an $H$-graded algebra, and $\cA$ is a central image (see \cite[\S 6]{ABFP}) of $\cB^{\tau^{-1}}$ (whose centroid is split).

\smallskip

If we want to get rid of the freedom in choosing $\tau\in\Zc^2_\textup{sym}(G,\FF^\times)$ mentioned in Remark \ref{re:freedom}, we may fix, for all subgroups $H$ of $G$, a section $\xi_H:\Hc^2_\textup{sym}(H,\FF^\times)\rightarrow \Hc^2_\textup{sym}(G,\FF^\times)$ of the map $\iota^*$ in \eqref{eq:longH}. Then we may consider the set $\overline{\mathfrak{A}}'(G,\FF)$  of triples $(H,[\tau'],[\cA])$, where $H$ is a subgroup of $G$, $[\tau']\in\Hc^2_\textup{sym}(H,\FF^\times)$, and 
$[\cA]$ is the equivalence class of a central simple and $G/H$-graded algebra $\cA$, under the equivalence relation being given by
$
\cA_1\sim \cA_2
$
if there is a character $\chi\in\Hom(H,\FF^\times)$ such that 
 $\bigl(\cA_1\bigr)_\chi\simeq_{G/H}\cA_2$.
 
\begin{corollary}\label{co:less_freedom}
The map
\[
\begin{split}
\overline{\mathfrak{A}}'(G,\FF)&\longrightarrow \overline{\mathfrak{B}}(G,\FF)\\
(H,[\tau'],[\cA])\ &\mapsto\ [L_\pi^\tau(\cA)]
\end{split}
\]
where $\tau$ is any $2$-cocycle in $\Zc^2_\textup{sym}(G,\FF^\times)$ such that $[\tau]=\xi_H([\tau'])$, is a bijection.
\end{corollary}
\begin{proof}
The map is clearly surjective. If two pairs $(H_1,[\tau_1'],[\cA_1])$ and $(H_2,[\tau_2'],[\cA_2])$ have the same image $[\cB]$, then the support of $C(\cB)$ is $H_1=H_2(\defby H)$, and its isomorphism class (as $H$-graded algebra) is determined by both $[\tau_1']$ and $[\tau_2']$ in $\Hc^2_\textup{sym}(H,\FF^\times)$. Hence $(H_1,[\tau_1'])=(H_2,[\tau_2'])$. Then $L_\pi^\tau(\cA_1)\simeq_G L_\pi^\tau(\cA_2)$, with $[\tau]=\xi_H([\tau_1'])=\xi_H([\tau_2'])$, which implies $L_\pi(\cA_1)\simeq_G L_\pi(\cA_2)$, and hence $\bigl(\cA_1)_\chi\simeq_{G/H}\cA_2$ for some $\chi\in\Hom(H,\FF^\times)$ by \cite[Theorem 7.1.1.(iii)]{ABFP}. Thus $(H_1,[\tau_1'],[\cA_1])=(H_2,[\tau_2'],[\cA_2])$.
\end{proof}

\bigskip

\section{The real case}\label{se:real}

If the ground field is the field of real numbers $\RR$, the $G$-graded-central-simple real algebras have been considered in \cite{BKpr}. The approach there is by Galois descent from $\CC$ to $\RR$. Hence any such algebra is obtained, up to isomorphism, by Galois descent from a complex loop algebra, and checked to be isomorphic to a \emph{loop algebra twisted by a character} $\chi:G\rightarrow S^1$, where $S^1$ denotes the unit circle in $\CC$.

Let us finish this section by showing the connection of these ``$\chi$-twisted loop algebras'' over $\RR$ with the cocycle twisted loop algebras considered here.

Given an abelian group $G$, the short exact sequence
\[
\begin{tikzcd}
\xi:\quad 1\arrow[r]&\RR^\times\arrow[r, hookrightarrow]&\CC^\times\arrow[r, "p"]&S^1\arrow[r]&1\\[-20pt]
&&z\arrow[r, mapsto] &z/\bar z&
\end{tikzcd}
\]
induces a long exact sequence
\begin{equation}\label{eq:longS1}
\begin{tikzcd}[column sep=small]
1\arrow[r, rightarrow]&\Hom(G,\RR^\times)\arrow[r]&\Hom(G,\CC^\times)\arrow[r]&\Hom(G,S^1)\arrow[r, "\delta"]
&\Ext(G,\RR^\times)\arrow[r]&1
\end{tikzcd}
\end{equation}
as $\Ext(G,\CC^\times)=1$ because $\CC^\times$ is divisible.

\begin{lemma}\label{le:real}
For any $\tau\in\Zc^2_\textup{sym}(G,\RR^\times)$ there is a character $\chi\in\Hom(G,S^1)$ such that, for any choice of elements $z_g\in\CC$ with $z_g^2=\chi(g)$ for all $g$, $[\tau]=[\dc \gamma]$ in $\Hc^2_\textup{sym}(G,\RR^\times)$, where $\gamma$ is the map $G\rightarrow S^1$, $g\mapsto z_g$. 
\end{lemma}
Note that $\dc\gamma(g_1,g_2)^2=\chi(g_1)\chi(g_2)\chi(g_1g_2)^{-1}=1$, so $\dc\gamma(g_1,g_2)\in\{\pm 1\}\subseteq \RR^\times$ and hence $\dc\gamma\in\Zc^2_\textup{sym}(G,\RR^\times)$, even though $\gamma$ takes values in $S^1\subseteq\CC^\times$.
\begin{proof}
The connecting homomorphism $\delta$ in \eqref{eq:longS1} works as follows. Given any character $\chi:G\rightarrow S^1$, let $E$ be the pull-back of $\chi$ and $p$:
\[
E=\{(z,g)\in\CC^\times\times G\mid \chi(g)=z/\bar z\},
\]
with its natural projections $\pi_1:E\rightarrow \CC^\times$ and $\pi_2:E\rightarrow G$. There appears the commutative diagram with exact rows:
\[
\begin{tikzcd}
\xi\chi:\quad 1\arrow[r]&\RR^\times\arrow[r]\arrow[d, equal]&E\arrow[r, "\pi_2"] \arrow[d, "\pi_1"]&G\arrow[r]\arrow[d, "\chi"]&1\\
\ \xi:\quad 1\arrow[r]&\RR^\times\arrow[r, hookrightarrow]&\CC^\times\arrow[r, "p"]&S^1\arrow[r]&1
\end{tikzcd}
\]
Then $\delta(\chi)$ is the class in $\Ext(G,\RR^\times)$ of $\xi\chi$. For any choice of elements $z_g\in \CC$ with $z_g^2=\chi(g)$, the map $G\rightarrow E$, $g\mapsto (z_g,g)$, is a section of $\pi_2$, and hence the element in $\Hc^2_\textup{sym}(G,\RR^\times)$ corresponding to the extension $\xi\chi$ is the class of the cocycle in $\Zc^2_\textup{sym}(G,\RR^\times)$ given by
\[
(g_1,g_2)\mapsto z_{g_1}z_{g_2}z_{g_1g_2}^{-1}\,.
\]
That is, identifying $\Ext(G,\RR^\times)$ with $\Hc^2_\textup{sym}(G,\RR^\times)$, $\delta(\chi)=[\dc\gamma]$ with $\gamma:g\mapsto z_g$. The Lemma now follows from the surjectivity of $\delta$.
\end{proof}

Now, given an abelian group $G$, a subgroup $H$, and a central simple and $G/H$-graded real algebra $\cA$, and given a character $\chi:G\rightarrow S^1$, the algebra $L_\pi^\chi(\cA)$ defined in \cite{BKpr} is the subalgebra $\bigoplus_{g\in G}\bigl(\cA_{\overline{g}}\otimes g\bigr)$ of $\cA\otimes_\RR \RR^{\dc\gamma}G$ ($\gamma$ as in Lemma \ref{le:real}), and hence it is precisely the $\dc\gamma$-twisted loop algebra $L_\pi^{\dc\gamma}(\cA)$ (see Remark \ref{re:Lpitau}). Lemma \ref{le:real} and Proposition \ref{pr:cocycle_twists} show then that for any $\tau\in\Zc^2_\textup{sym}(G,\RR^\times)$, there is a character $\chi:G\rightarrow S^1$ such that $L_\pi^\tau(\cA)\simeq_G L_\pi^\chi(\cA)$.

\bigskip

\begin{remark}\label{re:noncentral_complex}
Let $\cB$ be a central simple complex algebra, and denote by $\cB_\RR$ the real algebra obtained by restriction of scalars. That is, $\cB_\RR$ is just $\cB$, but considered as a real algebra. 
The results in this paper reduce the study of gradings on $\cB_\RR$ to gradings on the complex algebra $\cB$ and gradings on real forms of $\cB$, i.e., central simple real algebras $\cA$ such that $\cB$ is isomorphic to $\cA\otimes_\RR\CC$, as follows.

Let $G$ be an abelian group and let $\Gamma:\cB_\RR=\bigoplus_{g\in G}(\cB_\RR)_g$ be a $G$-grading on $\cB_\RR$. The centroid $C(\cB)=C(\cB_\RR)$ is $\CC 1$ ($1$ denotes here the identity map). Then either:

\begin{enumerate} 
\item The grading induced by $\Gamma$ on $C(\cB)$ is trivial: $C(\cB)=C(\cB)_e$. This means that each $(\cB_\RR)_g$ is a complex subspace of $\cB$, so $\Gamma$ is actually a $G$-grading of the complex algebra $\cB$: $\Gamma: \cB=\bigoplus_{g\in G}\cB_g$. 

Each isomorphism class of gradings $(\cB_\RR,\Gamma)$ corresponds to one or two isomorphism classes of gradings $(\cB,\Gamma)$, because the index $[\Aut(\cB_\RR):\Aut(\cB)]$ may be $1$ or $2$.

\smallskip

\item The grading induced by $\Gamma$ on $C(\cB)$ is not trivial. Then there is an element $h\in G$ of order $2$ such that $C(\cB)_e=\RR 1$, $C(\cB)_h=\RR\bi$ (identifying here $\bi$ with the map $x\mapsto \bi x$). Then $\cB_\RR$ is $G$-graded-central-simple and Theorem \ref{th:main} (and Corollaries \ref{co:main} and \ref{co:less_freedom}) applies, so that, with $H=\langle h\rangle$, we have an isomorphism $\cB_\RR\simeq_G L_\pi^\tau(\cA)$ for a central simple and $\overline{G}$-graded real algebra $\cA$ and a symmetric $2$-cocycle $\tau\in \Hc_\textrm{sym}(G,\RR^\times)$. Since $\CC^\times$ is divisible $\Hc^2_\textrm{sym}(G,\CC^\times)\simeq\Ext(G,\CC^\times)$ is trivial, so there is a map $\gamma:G\rightarrow \CC^\times$ such that $\tau=\dc\gamma$. Then the map:
\[
\Phi:L_\pi^\tau(\cA)\longrightarrow \cA\otimes_\RR\CC,\quad 
  x\otimes g\mapsto x\otimes \gamma(g)
\]
for $g\in G$ and $x\in \cA_{\overline{g}}$, is easily seen to be an isomorphism of $\overline{G}$-graded real algebras. Thus $\cB_\RR$ is isomorphic to $\cA\otimes_\RR\CC$ as real algebras, and composing, if necessary, with $\id\otimes(\text{complex conjugation})$, we check that $\cB$ and $\cA\otimes_\RR\CC$ are isomorphic complex algebras, so that $\cA$ is a real form of $\cB$.
\end{enumerate}
\end{remark}

\bigskip


\end{document}